\documentclass[12pt, reqno]{amsart} 
\usepackage{amssymb,amsthm,amsmath,mathtools}
\usepackage{layout}
\usepackage{enumerate}
\usepackage{ifthen}
\usepackage{graphicx} 
\usepackage[autostyle, english = american]{csquotes}
\usepackage{amsfonts}
\usepackage{amscd}
\usepackage{amssymb}
\allowdisplaybreaks
\usepackage{mathabx}
\usepackage{color}
\usepackage{amsbsy}
\usepackage{amsthm}
\usepackage{amsmath}
\usepackage{amsxtra}
\usepackage{mathrsfs}
\usepackage{bbm}
\usepackage{dsfont}

\usepackage{ifthen}

\usepackage[colorlinks,citecolor=red,hypertexnames=false]{hyperref} 
\usepackage[a4paper, margin=1.4in]{geometry}
\newtheorem{theorem}{Theorem}[section]
\newtheorem{cor}[theorem]{Corollary}
\newtheorem{lem}[theorem]{Lemma}

\theoremstyle{definition}
\newtheorem{defn}[theorem]{Definitions}
\theoremstyle{remark}
\newtheorem{remark}[theorem]{Remark}
\newtheorem{conjecture}[theorem]{Conjecture}
\numberwithin{equation}{section}
\definecolor{red}{rgb}{1.0, 0.0, 0.0}
\newcommand{\Fq} {{\mathbb F}_q}

\newcommand{\F} {\mathbb F}

\title[$\Fq$-primitive points on varieties over finite fields]
{$\Fq$-primitive points on varieties over finite fields}
\author{Soniya Takshak}
\address{Soniya Takshak \endgraf
	Department of Mathematics \endgraf
	Indian Institute of Technology Delhi \endgraf
	New Delhi, 110016, India} 
\email{sntakshak9557@gmail.com}
\author{Giorgos Kapetanakis}
\address{Giorgos Kapetanakis \endgraf
	Department of Mathematics \endgraf
	University of Thessaly \endgraf
	3rd km Old National Road
	Lamia-Athens, 35100 Lamia, Greece,} 
\email{gnkapet@gmail.com}
\author{Rajendra Kumar Sharma}
\address{Rajendra Kumar Sharma \endgraf
	Department of Mathematics \endgraf
	Indian Institute of Technology Delhi \endgraf
	New Delhi, 110016, India} 
\email{rksharmaiitd@gmail.com}

\keywords{Finite Field, Primitive Element, $r$-Primitive Element, Character} \subjclass[2020]{12E20, 11T23}
\date{\today}
\begin{document}
	\allowdisplaybreaks

	\begin{abstract} 
		
		Let $r$ be a positive divisor of $q-1$ and $f(x,y)$ a rational function of degree sum $d$ over $\Fq$ with some restrictions, where the degree sum of a rational function $f(x,y) = f_1(x,y)/f_2(x,y)$ is the sum of the degrees of $f_1(x,y)$ and $f_2(x,y)$. In this article, we discuss the existence of triples $(\alpha, \beta, f(\alpha, \beta))$ over $\Fq$, where $\alpha, \beta$ are primitive and $f(\alpha, \beta)$ is an $r$-primitive element of $\Fq$. In particular, this implies the existence of $\Fq$-primitive points on the surfaces of the form $z^r = f(x,y)$. As an example, we apply our results on the unit sphere over $\Fq$.  
			
	\end{abstract}
	\maketitle
%	\tableofcontents 
	\section{Introduction}
	%\noindent	
	Let $\Fq$ denote a finite field with $q$ elements. An element $\alpha$ in $\Fq$ is called primitive if it generates the cyclic group of non-zero elements of $\Fq$, denoted by $\Fq^*$. Clearly, the order of primitive elements is $q-1$. There are several applications of primitive elements in cryptography \cite{cryptographically-strong-sequences, note-on-discrete-logarithms, Paar2010}. Primitive elements are generalized as $r$-primitive elements for a positive divisor $r$ of $q-1$. More precisely, an element $\alpha \in \Fq$ is called $r$-primitive if its order is $\frac{q-1}{r}$. Although primitive elements are elusive, there are, in some cases, effective ways to find $r$-primitive elements \cite{Elements-of-high-order, Constructing-FF-LargeOrderElements, Constructing-high-order-elements}. Many articles in the literature \cite{booker2018primitive, s.d, WANG2012800, SoniyaTakshak} deal with the existence of primitive pairs $(\alpha, f(\alpha))$ over $\Fq$, where $f(x)$ is a non-exceptional rational function in $\Fq(x)$. This is equivalent to the existence of $\Fq$-primitive points on the curves $y=f(x)$ over $\Fq$.
	%\color{blue}
	Note that, the \textquote{$\Fq$-primitivity} on curves is not related to the \textquote{$\Fq$-primitivity} on elliptic curves defined in earlier works. More precisely, in \cite{Primitive-points-on-elliptic-curves}, Lang and Trotter called a point to be $\Fq$-primitive point on an elliptic curve if it generates the elliptic curve with respect to the elliptic curve group operation.
	%\color{black}
	In 2022, Cohen, Kapetanakis and Reis \cite{The-existence-of-Fq-primitive-points-on-curves-using-freeness} generalized the notion of freeness and proved the existence of $\Fq$-primitive points on the curves of the form $y^n = f(x)$ over $\Fq$.
	In this article, we extend the notion of exceptional functions from one variable to the functions of $n$-variables and call them \emph{primarily exceptional} functions. Then  we obtain a sufficient condition for the existence of triples of the form $(\alpha, \beta, f(\alpha,\beta))$ over $\Fq$, where $\alpha, \beta$ are primitive and $f(\alpha, \beta)$ is $r$-primitive for a primarily non-exceptional rational function of degree sum $d$ over $\Fq$. In other words, we discuss the existence of $\Fq$-primitive points on the surfaces of the form $z^r= f(x,y)$, where $r|q-1$ and $f(x,y)$ is a primarily non-exceptional rational function over $\Fq$.
	
	\section{Preliminaries}
	%\noindent
	Let us recall some basic concepts crucial for our proof. For positive integers $a$ and $b$, set $a_{(b)} = \frac{a}{\text{gcd}(a,b)}$. For a group $G$, a character $\chi$ is a homomorphism from $G$ to the multiplicative group of complex numbers with unit modulus. A character $\chi_1$ is called the trivial character of $G$ when $\chi_1(g) = 1$ for all $g \in G$. Two types of characters can be defined over a finite field $\Fq$ corresponding to the operations of addition and multiplication, namely additive and multiplicative characters, respectively. A multiplicative character $\chi$ can be extended from $\Fq^*$ to $\Fq$ by the following rule
	$$ \chi(0) \coloneqq \begin{cases}
		0,~~~ \text{if}~~~ \chi \ne \chi_1,\\
		1,~~~ \text{if}~~~ \chi = \chi_1.
	\end{cases}$$
	\noindent
	The order of a character $\chi$ is defined to be the smallest positive integer $s$ such that $\chi^s = \chi_1$. We denote the set of multiplicative characters of $\F_{q}$ by $\widehat{\F_{q}}$.
	
	An element $\alpha \in \Fq$ is called $e$-free for $e|q-1$, if $\alpha$ can not be written as $\beta^d$ for any $\beta \in \Fq$ and $d|e$ with $d > 1$. Clearly, $\alpha$ is primitive when $\alpha$ is $(q-1)$-free. Following Cohen and Huczynska \cite{cohen.h, cohen.s}, the characteristic function for the set of $e$-free elements of ${\mathbb{F}_{q}^{*}}$, where $e|q-1$, is given by
	$$\rho_e:
	\alpha \mapsto \theta(e) \sum\limits_{s|e} \frac{\mu(s)}{\phi(s)} \sum\limits_{\chi_s}\chi_s(\alpha),$$
	where $\theta(e) := \frac{\phi(e)}{e},$ $\mu$ denotes the M\"obius function, and $\chi_s$ is a multiplicative character of order $s$ in $\widehat{\F_{q}}$.
	
	For a divisor $r$ of $q-1$ and a divisor $R$ of $\frac{q-1}{r}$, let $C_r$ denote the multiplicative subgroup of $\F_{q}^*$ of order $\frac{q-1}{r}$, an element $\alpha \in \F_{q}^*$ is called $(R, r)$-free if $\alpha \in C_r$ and $\alpha$ is $R$-free in $C_r$, i.e., if $\alpha$ can not be written as $\alpha = \beta^s$ for any $\beta \in C_r$ and $s|R$, with $s > 1$. Clearly, $\alpha$ is $r$-primitive when $\alpha$ is $(\frac{q-1}{r}, r)$-free. The characteristic function for $r$-primitive elements was first expressed in a useful way by Cohen and Kapetanakis \cite{line-or-translate-property}. Later, Cohen et al. \cite{The-existence-of-Fq-primitive-points-on-curves-using-freeness} provided an improved version of the characteristic function as follows. 
	%Let $\mathbb{I}_{R,r}$ be the characteristic function of $(R, r)$-free elements of $\F_{q^n}^*$, i.e. 
	%$$\mathbb{I}_{R,r}(\alpha)=
	%\begin{cases}
	%1, & \text{if}\ \alpha\ \text{is}\ $(R, r)$-\text{free},\\
	%0, & \text{otherwise}.	
	%\end{cases}$$
	$$ \mathbb{I}_{R,r}: \alpha \mapsto \frac{\theta(R)}{r}\sum_{d|Rr}^{} \frac{\mu(d_{(r)})}{\phi(d_{(r)})}\sum_{\chi_d}^{} \chi_d(\alpha),$$	
	where $\theta(R) = \frac{\phi(R)}{R},$ and $\chi_d$ is a	multiplicative character of $\Fq^*$ of order $d$. 
	
	To bound certain character sum that will arise, we will use the following result.
	
	\begin{lem}\emph{\cite{The-existence-of-Fq-primitive-points-on-curves-using-freeness}}
		For any positive integers $R, r$, we have that
		$$ \sum_{d|R}^{} \frac{|\mu(d_{(r)})|}{\phi(d_{(r)})}\cdot \phi(d) = \text{gcd}(R, r) \cdot W(\text{gcd}(R, R_{(r)})).$$
	\end{lem}
	\noindent
	The following result will help us prove the sufficient condition in Theorem \ref{sufficient_condition}. 
	
	\begin{lem}\emph{\cite{Stepnov}}\label{Stepnov}
		Let $F(x) \in \mathbb{F}_q(x)$ be a rational function. Write $F(x)=\prod_{j=1}^{k}F_j(x)^{r_j}$, where $F_j(x) \in \mathbb{F}_{q}[x]$ are irreducible polynomials and $r_j$ are nonzero integers. Let $\chi$ be a multiplicative character of $\F_{q}$ of order $d$ (a divisor of $q-1$). Suppose that $F(x)$ is not of the form $cG(x)^d$ for any rational function $G(x) \in \mathbb{F}_q(x)$ and $c \in \mathbb{F}_q*$. Then we have 
		$$\left|\sum\limits_{\alpha \in \F_{q},F(\alpha)\neq \infty} \chi(F(\alpha))\right| \leq \left( \sum\limits_{j=1}^{k}deg(F_j) -1\right)q^{1/2}.$$
	\end{lem}  
	
	\section{Sufficient Condition}
	First, we recall exceptional functions of one variable over $\Fq$ introduced by Cohen et al. \cite{COHEN2021237}. A function $g(x) \in \Fq(x)$ is called an exceptional rational function if $g(x) = \lambda h(x)^s$, for some $\lambda \in \Fq^*$, $h(x) \in \Fq(x)$ and $s|q-1$ such that $s > 1$. Now, we extend this notion to the functions of $n$ variables over $\Fq$.
	
	\begin{defn}
		A function $f(x_1,x_2,\dots,x_n) \in \F_q(x_1,x_2,\dots,x_n)$ is called primarily exceptional rational function if $f(\alpha_1, \alpha_2,\dots,\alpha_{n-1},x_n)$ are exceptional functions of one variable $x_n$ over $\Fq$ for every possible primitive tuple $(\alpha_1, \alpha_2, \dots, \alpha_{n-1})$ over $\Fq$.
	\end{defn}
	
	In particular, $f(x,y) \in \F_q(x,y)$ is called \emph{primarily exceptional} rational function if $f(\alpha,y)$ are exceptional functions for every primitive $\alpha \in \F_q$. A stronger definition could also be given by taking both $f(\alpha,y)$ and $f(x,\alpha)$ to be exceptional for every primitive $\alpha \in \F_q$. In this work, for the sake of convenience, we will work with the former definition. However, without loss of generality, our results are valid even if the latter definition of primarily exceptional functions is chosen.
	%\noindent
	Clearly, functions of the form $cg(x,y)^d$ are trivial examples of \emph{primarily exceptional} functions for $d | q-1$, $c \in \Fq^*$ and $g(x,y) \in \Fq(x,y)$. This family of functions gives rise to the exceptional functions of one variable $f(a,y)$ for every $a\in\F_q$ (not necessarily primitive). 
	However, a function of the form 
	$$f(x,y) = \prod\limits_{\substack{\alpha \in \F_{q}\\ \alpha \ \text{primitive}}}^{}(x-\alpha)g(x,y) + c\cdot h(x,y)^d \in \F_{q}(x,y)$$ 
	will give rise to the exceptional functions of one variable $f(\alpha,y)$ but for generic $a\in\F_q, f(a,y)$ may not be exceptional, if $g$ and $h$ are chosen carefully.
	
	Let $\Delta_q(d)$ denote the set of rational functions $f(x,y) \in \F_q(x,y)$ of degree sum $d$ that are not primarily exceptional over $\Fq$. For a positive integer $\ell$, we denote the number of square-free divisors of $\ell$ by $W(\ell)$.
	To show the existence of primitive points over $z^r=f(x,y)$ over $\Fq$, we fix $\alpha$ to be a primitive element of $\Fq$. Then, we check whether $f(\alpha,\beta)$ is $r$-primitive when $\beta \in \Fq$ is primitive. For $\ell|q-1$ and $R|\frac{q-1}{r}$, let $N_{f}(\ell,R)$ represent the number of triples $(\alpha, \beta, f(\alpha,\beta))$ over $\Fq$ such that $\alpha$ is primitive, $\beta$ is ${\ell}$-free and $f(\alpha, \beta)$ is $(R,r)$-free. Let $\mathfrak{B}_{f}$ denote the set of prime power $q$ such that there exists a primitive point on the surface $z^r = f(x,y))$ for $f(x,y) \in \Delta_q(d)$. Hence, to show that $q \in \mathfrak{B}_f$, it is sufficient to prove that $N_{f}(q-1, \frac{q-1}{r}) > 0$.
	
	\begin{theorem}\label{sufficient_condition}
		Let $q$ be a prime power, $r| q-1, d \in \mathbb{N}$ and $f(x,y) \in \Delta_q(d)$. Then $q \in \mathfrak{B}_f$ if
		\begin{equation}\label{condition}
			q^{1/2} > rdW(q-1)W\left(\frac{q-1}{r}\right).
		\end{equation}
	\end{theorem}
	\begin{proof}
		For primitive $\alpha \in \Fq$, $\mathcal{V}$ denotes the set of poles of $f(\alpha^k, y)$ in $\Fq$, where $1\leq k\leq q-1$ is such that gcd$(k,q-1)=1$.
		Now, for $R| \frac{q-1}{r}$ and $l| q-1$,
		\begin{equation*}
			\begin{aligned}
				N_{f}(\ell, R) =  & \sum_{\substack{1\leq k\leq q-1 \\ \gcd(k,q-1)=1}} \sum_{\beta \in \mathbb{F}_{q} \backslash \mathcal{V}} \rho_{\ell}(\beta) \mathbb{I}_{R,r} f(\alpha^k,\beta)\\
				= {} & \frac{\theta(\ell)\theta(R)}{r} \sum\limits_{\substack{d_1|{\ell},d_2|Rr}}^{}\frac{\mu(d_1)\mu(d_{2(r)})}{\phi(d_1)\phi(d_{2(r)})} \sum\limits_{\substack{\chi_{d_1},\chi_{d_2}}}^{}\mathcal{S}_{d_1,d_2}
			\end{aligned}
		\end{equation*}
		\noindent
		where $\mathcal{S}_{d_1,d_2} =
		\sum\limits_{\substack{1\leq k\leq q-1 \\\gcd(k,q-1)=1}}\sum\limits_{\beta \in \mathbb{F}_{q} \backslash \mathcal{V}}^{} \chi_{d_1}(\beta)\chi_{d_2}(f(\alpha^k, \beta))$.\\	
		For multiplicative character $\chi_{q-1}$ of order $q-1$, $ \chi_{d_j}(\beta) = \chi_{q-1}(\beta^{n_j})$, for some $n_j \in \{0,1,2, \ldots ,q-2\},j=1,2$. Therefore
		
		\begin{equation*}
			\mathcal{S}_{d_1,d_2} = \sum\limits_{\substack{1\leq k\leq q-1 \\\gcd(k,q-1)=1}}\sum\limits_{\beta \in \mathbb{F}_{q} \backslash \mathcal{V}}^{} \chi_{q-1}(\beta^{n_1}f(\alpha^k,\beta)^{n_2})
		\end{equation*}
		
		If $y^{n_1}f(\alpha^k,y)^{n_2} \neq cG(y)^{q-1}$ for any $G(y) \in \mathbb{F}_q(y)$ and $c \in \mathbb{F}_q^*$, then using Lemma \ref{Stepnov}, we obtain
		$$|\mathcal{S}_{d_1,d_2}| \leq d\phi(q-1)q^{1/2}.$$
		If $y^{n_1} f(\alpha^k, y)^{n_2} = cG(y)^{q-1}$, write $f(\alpha^k, y) = \omega y^t\prod_{j=1}^{n} f_{j}(y)^{s_j}$, where $n$ is the number of irreducible factors of $f(\alpha^k,y)$ and $s_j$ is the multiplicity of those factors $f_{j}(y)$. Suppose $s_j'$ denote the multiplicity of $f_j(y)$ in $G(y)$. Comparing the multiplicity of $f_j(y)$ on both sides, we get $s_jn_2 = s_j'(q-1)$.
		Let $s = (q-1)/\text{gcd}(q-1,n_2)$. Clearly $s|s_j$ and $s|q-1$. Hence $f(\alpha^k, y) = \omega y^t g(y)^s$ for some $g(y) \in \Fq(y)$. Since $f(x, y) \in \Delta_q(d)$, we get $s=1$. This implies $q-1|n_2$, hence $n_2=0$. Therefore $G(y)= y^t$ for some positive integer $t$ and we get $n_1= t(q-1)$. This implies $n_1=0$. In other words $d_1=1$ and $d_2=1$.
		Therefore, when $(d_1,d_2) \neq (1,1)$ we have
		$$|\mathcal{S}_{d_1,d_2}| < d\phi(q-1)q^{1/2}.$$
		Hence
		\begin{align*}\label{conditionLastEqn}
			N_{f}(\ell, R)  & \ge {} \frac{\theta(\ell)\theta(R)}{r}\phi(q-1)\{ q - |\mathcal{V}| - d q^{\frac{1}{2}} (rW(\ell)W(R)-1)\} \nonumber \\
			& > {} \frac{\theta(\ell)\theta(R)}{r}\phi(q-1) \{q - rd q^{\frac{1}{2}}W(\ell)W(R) \}.
		\end{align*}	
		\noindent
		Hence $N_{f}(\ell,R) > 0$, whenever $q^{\frac{1}{2}} > rd W(\ell)W(R).$ Now, the result follows by taking $\ell = q - 1$ and $R= \frac{q-1}{r}$.
	\end{proof}
	
	%\color{blue}
	\begin{cor}
		For any $d$ and $r$, there exists some $C_{d,r}$, such that for any prime power $q>C_{d,r}$ with $r\mid q-1$ and $f\in \Delta_q(d)$, there exists an $\F_q$-primitive point on the surface $z^r=f(x,y)$.
	\end{cor}
	
	%\color{black}
	\begin{remark}
		Since $1$-primitive element is simply a primitive element in $\Fq$, $r=1$ in Theorem \ref{sufficient_condition} gives the sufficient condition for the existence of primitive triples $(\alpha, \beta, f(\alpha, \beta))$ over $\Fq$. The same result applies to the existence of primitive $n$-tuple of the form $(\alpha_1, \alpha_2, \dots, \alpha_{n-1}, f(\alpha_1, \alpha_2, \dots, \alpha_n))$ over $\Fq$, where $f(x_1, x_2, \dots, x_n)$ is primarily non-exceptional rational function over $\Fq$ and $(\alpha_1, \alpha_2, \dots, \alpha_{n-1})$ is primitive $n-1$-tuple over $\Fq$.
	\end{remark}

	The sufficient condition (\ref{condition}) is improved using the Cohen-Huczynska sieving technique \cite{cohen.h}. The proof follows from the idea of \cite[Theorem 3.4]{gupta} and is omitted.
	\begin{theorem}\label{improved}
		Let $\ell|q-1$ and $p_1, p_2, \dots p_{t_1}$ are all primes dividing $q-1$ but not $\ell$, and $\ell'|\frac{q-1}{r}$ and $p_1', p_2', \dots p_{t_2}'$ are all primes dividing $\frac{q-1}{r}$ but not $\ell'$. Take $\delta = 1 - \sum_{i=1}^{t_1}\frac{1}{p_i} - \sum_{i=1}^{t_2}\frac{1}{p_i'}$ and $\mathcal{S} = \frac{t_1 + t_2 - 1}{\delta} + 2.$ Assume $\delta > 0$, then $N_{f}(q-1, \frac{q-1}{r}) > 0$ if 
		\begin{equation}\label{sieve}
			q^{1/2} > rd \mathcal{S} W(\ell)W(\ell').
		\end{equation}
		Hence $q \in \mathfrak{B}_f$.	
	\end{theorem}
	
	\section{Estimating $C_{d,r}$}
	In this section, we provide a general process of finding the estimated value of $C_{d,r}$. The following Lemma is very useful in the coming discussion.
	
	\begin{lem}\emph{\cite{cohens}}\label{bound on C}
		For a positive integer $n \in \mathbb{N}, W(n) \leq K n^{1/6}$, where $K = \frac{2^t}{(p_1 p_2 \cdots p_t)^{1/6}}$ and $p_1, p_2, \dots, p_t$ are distinct primes less than $2^6$. 
		%In particular, when $n \in \mathbb{N}, \mathcal{C} < 37.469$ and when $n$ is odd $\mathcal{C}< 21.029$.
	\end{lem}
	
	It is clear from Lemma \ref{bound on C} that $W(\frac{q-1}{r}) \leq W(q-1) < 37.469 q^{1/6}$. Hence the sufficient condition (\ref{condition}) becomes $q^{1/2} > 1403.926 rdq^{1/3}.$ This implies $q^{1/6} > 1403.926 rd$. Hence $q \in \mathfrak{B}_f$ if $q > 7.65712 \times 10^{18} (rd)^6$.
	
	\subsection{When $\mathbf{r=2}$ and $\mathbf{d=2}$.}
	Clearly, $r=2$ is only possible when $q$ is odd. Now we split our discussion into two parts depending upon $q \equiv 1\ \text{mod}\ 4$ or $q \equiv 3\ \text{mod}\ 4$.
	
	\subsubsection{When $q \equiv 1\ \emph{mod}\ 4$.}
	Using Lemma \ref{bound on C}, $W(\frac{q-1}{2}) = W(q-1) < 37.469 q^{1/6}$. Hence the sufficient condition (\ref{condition}) becomes $q > 3.137 \times 10^{22}$. When $q-1$ is divided by the first $18$ primes, $q > 3.137 \times 10^{22}$ holds. Now we use Theorem \ref{improved} and apply the sieving technique to reduce this further. Suppose $5 \leq w(q-1) \leq 17$ and take $\ell, \ell'$ to be the product of the least $5$ primes. Hence $W(\ell) = W(\ell') = 2^5$ and $t_1, t_2 \leq 12$. The value of $\delta$ will be smallest when $\{p_1, p_2, \dots, p_{t_1}\} = \{13,17,\dots, 59\}$. This yields $\delta > 0.139272, \mathcal{S}< 167.14453$ and $4\mathcal{S}W(\ell)W(\ell') \leq 6.85 \times 10^5$. That is $q \in \mathfrak{B}_f$ if $q^{1/2}> 6.85 \times 10^5$ or $q > 4.69 \times 10^{11}$, and this holds when $w(q-1) \geq 12$. Next assume $4 \leq w(q-1) \leq 11$ and $\ell, \ell'$ to be the product of first $4$ primes. That is $W(\ell)= W(\ell')= 2^4$. The set of primes $\{p_1, p_2, \dots, p_{t_1}\} = \{11, 13, \dots, 31\}$ gives $\delta > 0.2209873, \mathcal{S} < 97.0280941, 4 \mathcal{S} W(\ell) W(\ell') \leq 9.936 \times 10^4$. That is $q \in \mathfrak{B}_f$ if $q^{1/2}> 9.936 \times 10^4$ or $q > 9.8724 \times 10^{9}$. We follow the same procedure again and obtain that $q \in \mathfrak{B}_f$ when $q > 4.072 \times 10^9$.
	
	\subsubsection{When $q \equiv 3\ \emph{mod}\ 4$.}
	In this case, $W(\frac{q-1}{2}) = \frac{W(q-1)}{2} < \frac{37.469}{2}q^{1/6}$. Using Lemma \ref{bound on C} and condition (\ref{condition}), $q \in \mathfrak{B}_f$ if $q > 4.901 \times 10^{20}$. This is true when $w(q-1) \geq 17$. Now assume $4 \leq w(q-1) \leq 16$ and $4 \leq w(\frac{q-1}{2}) \leq 15$. Take $\ell, \ell'$ to be the product of the first $4$ primes. That is $W(\ell) = W(\ell')= 2^4, t_1=12, t_2= 11$. Setting $\{p_1, p_2, \dots, p_{t_1}\}= \{11, 13, \dots, 53\}$ and $\{p_1, p_2, \dots, p_{t_2}\}= \{13,17, \dots, 53\}$ yields $\delta > 0.0823, \mathcal{S} < 366.6922, 4 \mathcal{S} W(\ell) W(\ell') < 3.755 \times 10^5$. Hence $q \in \mathfrak{B}_f$ if $q^{1/2} > 3.755 \times 10^5$ or $q > 1.41 \times10^11$. This holds when $w(q-1) \geq 11$. 
	Next, assume $3 \leq w(q-1) \leq 10$ and $3 \leq w(\frac{q-1}{2}) \leq 9$ and take $\ell, \ell'$ to be the product of the first $3$ primes. That is $W(\ell) = W(\ell')= 2^3, t_1= 7, t_2=6$. This yields $\delta > 0.14265, \mathcal{S} < 128.1863, 4 \mathcal{S} W(\ell) W(\ell') < 3.282 \times 10^4$. Hence $q \in \mathfrak{B}_f$ if $q^{1/2} > 3.282 \times 10^4$ or $q > 1.078 \times10^9$, that holds when $w(q-1) \geq 10$. We follow the same procedure again and obtain that $q \in \mathfrak{B}_f$ when $q > 9.026 \times 10^8$.
	
	In this way, we obtain that $C_{2,2} \leq 4.072 \times 10^9$ when $q \equiv 1\ \text{mod}\ 4$ and $C_{2,2} \leq 9.026 \times 10^8$ when $q \equiv 3\ \text{mod}\ 4$.
	
	\begin{remark}
		We proved that $C_{2,2} \leq 4.072 \times 10^9$. This large number suggests that identifying, or even estimating $C_{d,r}$ for $d>2$ is an intriguing and non-trivial task.
		We applied our result to the unit sphere $z^2 = 1 - x^2 - y^2$ over $\Fq$ and checked the existence of $\Fq$-primitive points using GAP algebra system \emph{\cite{GAP4}}. Our computations verify the existence of $\Fq$-primitive points for all values of $q$ except $q=3, 5, 9, 13$ and $25$.
	\end{remark}
	On the basis of our computational results, we propose the following conjecture regarding the unit sphere over $\Fq$.
	
	\begin{conjecture}
		Let $q$ be a prime power such that $q > 25$ and $x^2 + y^2 + z^2 = 1$ be the unit sphere over $\Fq$. Then there exists an $\Fq$-primitive point on the unit sphere. 
	\end{conjecture}

	%		\section{Data availability statement}
	%		The authors confirm that the data supporting the findings of this study are available within the article  and its supplementary materials.
	%		
	%		
	%		\section{Declarations}\vspace{0.1cm}
	%	 \textbf{Ethical Approval:} Not applicable.\vspace{0.1cm}
	%		
	%		
	%		\textbf{Competing interests:} No potential competing of interest was reported by the author.\vspace{0.1cm}
	%		
	%		
	%		
	% \textbf{Authors' contributions:} 		 All the authors are contributed equally.\vspace{0.1cm}
	%
	%
	%	\textbf{Funding:}  The first and third authors were supported by Core Research Grant(RP03890G), Science and Engineering Research Board (SERB), DST, India.  The second author was supported  by the FWO Odysseus 1 grant G.0H94.18N: Analysis and Partial Differential Equations, the Methusalem programme of the Ghent University Special Research Fund (BOF) (Grant number 01M01021), and by FWO Senior Research Grant G011522N.\vspace{0.1cm}
	%	
	%	
	%		\textbf{Availability of data and materials:} All the data uned are  within the manuscript.


\begin{thebibliography}{99}
		
		
		%\baselineskip=12pt
		%\baselineskip=17pt
		
		%	\bibitem{Ahmad15} B. Ahmad, A. Alsaedi and M. Kirane, Nonexistence of global solutions of some nonlinear space-nonlocal evolution equations on the Heisenberg group, {\it Electron. J. Differ. Equ.}, 2015(227), 1–10 (2015).
		
		%	\bibitem{AZ22}		J.-Ph. Anker and H.-W. Zhang, Wave equation on general noncompact symmetric spaces, (to appear in) {\em Amer. J. Math}, 2022. arXiv:2010.08467
		
		
		%	\bibitem{EE1} D. Applebaum, L\'evy processes from probability to finance quantum groups, \emph{Notices Amer. Math. Soc.} 51, 1336-1347 (2004).
		
		%		\bibitem{EE3}  G. Autuori and P. Pucci, Elliptic problems involving the fractional Laplacian in $\mathbb{R}^n $, \emph{J. Differential Equations} 255(8),  2340-2362 (2013).
	{\normalsize	
		\bibitem{cryptographically-strong-sequences} Manuel Blum and Silvio Micali, How to generate cryptographically strong sequences of pseudo random bits, In {\it 23rd Annual Symposium on Foundations of Computer Science} (sfcs 1982), 112–117, (1982).
		
		
		\bibitem{booker2018primitive} Andrew Booker, Stephen Cohen, Nicole Sutherland, and Tim Trudgian, Primitive values of quadratic polynomials in a finite field, {\it Mathematics of Computation}, 88, (2019).
		
		
		\bibitem{Elements-of-high-order} F.E. Brochero Mart\'{i}nez and Lucas Reis, Elements of high order in Artin–Schreier extensions of finite fields $\Fq$, {\it Finite Fields and Their Applications},
		41, 24–33, (2016).
		
		
		\bibitem{Constructing-FF-LargeOrderElements} Qi Cheng, Constructing finite field extensions with large order elements, {\it SIAM Journal on Discrete Mathematics}, 21(3), 726–730, (2007).
		
		
		\bibitem{Constructing-high-order-elements} Qi Cheng, Shuhong Gao, and Daqing Wan, Constructing high order elements through subspace polynomials, In {\it Proceedings of the Annual ACM-SIAM Symposium on Discrete Algorithms}, 1457–1463, (2012).
		
		
		\bibitem{Stepnov} Todd Cochrane and Christopher Pinner, Using stepanov’s method for exponential sums involving rational functions, {\it Journal of Number Theory}, 116(2), 270–292, (2006).
		
		
		\bibitem{cohens} S. D. Cohen, Pairs of primitive elements in fields of even order, {\it Finite Fields Appl.}, 28, 22–42, (2014).
		
		
		\bibitem{cohen.h} S. D. Cohen and S. Huczynska, The primitive normal basis theorem-without a computer, {\it J. Lond. Math. Soc.}, 67(1), 41–56, (2003).
		
		
		\bibitem{cohen.s} S. D. Cohen and S. Huczynska, The strong primitive normal basis theorem, {\it Acta Arith.}, 143(4), 299–332, (2010).
		
		
		\bibitem{s.d} Stephen Cohen. Consecutive primitive roots in a finite field. {\it Proceedings of the American Mathematical Society}, 93, 189–197, (1985).
		
		
		\bibitem{line-or-translate-property} Stephen Cohen and Giorgos Kapetanakis, Finite field extensions with the line or translate property for $r$-primitive elements, {\it J. Aust. Math. Soc.}, 111, 313–319, (2021).
		
		
		\bibitem{The-existence-of-Fq-primitive-points-on-curves-using-freeness} Stephen Cohen, Giorgos Kapetanakis, and Lucas Reis, The existence of $\Fq$-primitive points on curves using freeness, {\it Comptes Rendus Math\'{e}matique}, 360, 641–652, (2022).
		
		
		\bibitem{COHEN2021237} Stephen D. Cohen, Hariom Sharma, and Rajendra Sharma, Primitive values of rational functions at primitive elements of a finite field, {\it Journal of Number
		Theory}, 219, 237–246, (2021).
		
		
		\bibitem{GAP4} The GAP Group. {\it GAP – Groups, Algorithms, and Programming}, Version 4.13.1, (2024).
		
		
		\bibitem{gupta} Anju Gupta, R. K. Sharma, and S. D. Cohen, Primitive element pairs with one prescribed trace over a finite field, {\it Finite Fields Appl.}, 54, 1–14, (2018).
		
		
		\bibitem{Primitive-points-on-elliptic-curves} S. Lang and H. Trotter, Primitive points on elliptic curves, {\it Bull. Am. Math. Soc.}, 83, 289–292, (1977).
		
		
		\bibitem{note-on-discrete-logarithms} G. Meletiou and Gary L. Mullen, A note on discrete logarithms in finite fields, {\it Applicable Algebra in Engineering, Communication and Computing}, 3(1), 75–78, (1992).
		
		
		\bibitem{Paar2010} Christof Paar and Jan Pelzl, Public-Key cryptosystems based on the discrete logarithm problem, {\it Springer Berlin Heidelberg, Berlin, Heidelberg} 205–238. (2010).
		
		
		\bibitem{SoniyaTakshak} R. K. Sharma, Soniya Takshak, Ambrish Awasthi and Hariom Sharma, Existence of rational primitive normal pairs over finite fields, {\it International Journal of Group Theory}, 13(1), 17--30, (2024).
		
		
		\bibitem{WANG2012800} Peipei Wang, Xiwang Cao, and Rongquan Feng, On the existence of some specific elements in finite fields of characteristic $2$, {\it Finite Fields and Their Applications}, 18(4), 800–813, (2012).
	}
		
		%\bibitem{AZ} J.-Ph. Anker and H.-W. Zhang, Wave equation on general noncompact symmetric spaces, (to appear in) {\it Amer. J. Math}, (2024). \url{ https://doi.org/10.48550/arXiv.2010.08467}.
		 
	
		%\bibitem{DKM23} A. Dasgupta, V. Kumar and S. S. Mondal, Nonlinear fractional damped wave equations on compact Lie groups, \emph{Asymptot. Anal.} (2023). \url{https://doi.org/10.3233/ASY-231842} 	
		
		%\bibitem{GKR} S. Ghosh, V. Kumar and M. Ruzhansky, Compact Embeddings, Eigenvalue Problems, and subelliptic Brezis-Nirenberg equations involving singularity on stratified Lie groups, \emph{Math. Ann.} (2023). arxiv preprint. \url{https://doi.org/10.1007/s00208-023-02609-7 }
		
		
		%\bibitem{GKR2} S. Ghosh, V. Kumar and M. Ruzhansky, Best constants in subelliptic fractional  Sobolev and Gagliardo-Nirenberg inequalities and ground states on stratified Lie groups, (2023). \url{https://doi.org/10.48550/arXiv.2306.07657 }
		
		
		%\bibitem{31}      A.  Palmieri, Semilinear wave equation on compact Lie groups, \emph{J. Pseudo-Differ. Oper. Appl.} 12, 43 (2021).
		
		
	\end{thebibliography}
\end{document}